\documentclass[12pt]{amsart}
\usepackage{amsfonts,amssymb,amscd,amstext,mathrsfs}

\input xy
\xyoption{all}

\newtheorem{theorem}{Theorem}

\newtheorem{corollary}[theorem]{Corollary}
\newtheorem{proposition}[theorem]{Proposition}

\newtheorem*{theorem*}{Theorem}
\newtheorem*{main-theorem}{Main Theorem}

\theoremstyle{definition}
\newtheorem{remark}[theorem]{Remark}

\newcommand{\A}{\mathscr{A}}
\newcommand{\C}{\mathbb{C}}
\renewcommand{\P}{\mathbb{P}}
\newcommand{\codim}{\operatorname{codim}}
\renewcommand{\ker}{\operatorname{Ker}}

\begin{document}

\title{Algebraic subellipticity and dominability \\ of blow-ups of affine spaces}

\author{Finnur L\'arusson and Tuyen Trung Truong}

\address{School of Mathematical Sciences, University of Adelaide, Adelaide SA 5005, Australia}
\email{finnur.larusson@adelaide.edu.au}
\email{tuyen.truong@adelaide.edu.au}

\thanks{The authors were supported by Australian Research Council grants DP120104110 and DP150103442.}


\date{1 July 2016.  Minor changes 4 November 2016}


\begin{abstract} 
Little is known about the behaviour of the Oka property of a complex manifold with respect to blowing up a submanifold.  A manifold is of Class $\A$ if it is the complement of an algebraic subvariety of codimension at least $2$ in an algebraic manifold that is Zariski-locally isomorphic to $\C^n$.  A manifold of Class $\A$ is algebraically subelliptic and hence Oka, and a manifold of Class $\A$ blown up at finitely many points is of Class $\A$.  Our main result is that a manifold of Class $\A$ blown up along an arbitrary algebraic submanifold (not necessarily connected) is algebraically subelliptic.  For algebraic manifolds in general, we prove that strong algebraic dominability, a weakening of algebraic subellipticity, is preserved by an arbitrary blow-up with a smooth centre.  We use the main result to confirm a prediction of Forster's famous conjecture that every open Riemann surface may be properly holomorphically embedded into $\C^2$.
\end{abstract}

\maketitle

\section{Introduction and Results} 
\label{sec:intro}

\noindent
Modern Oka theory has evolved from Gromov's seminal work on the Oka principle \cite{Gromov1989}.  (The monograph \cite{Forstneric2011} is a comprehensive reference on Oka theory.  See also the surveys \cite{Forstneric2013} and \cite{FL2011}.)  Oka theory may be viewed as the study of approximation and interpolation problems for holomorphic maps from Stein spaces into suitable complex manifolds.  The goal, for suitable targets, is to show that such a problem can be solved as soon as there is no topological obstruction to its solution.  The suitable targets turn out to be the so-called Oka manifolds.  From another point of view, Oka theory is the study of complex manifolds that are the targets of \textit{many} holomorphic maps from Stein spaces, with the word \textit{many} interpreted homotopically.  The fundamental result in this direction is that every continuous map from a Stein space to an Oka manifold can be deformed to a holomorphic map.  From a third point of view, Oka theory is seen as an answer to the question:  What is a good definition of \textit{anti-hyperbolicity} for complex manifolds?  

The prototypical examples of Oka manifolds are complex Lie groups and their homogeneous spaces.  Among other known examples are manifolds of the so-called Class $\A$.  A manifold is of Class $\A$ if it is the complement of an algebraic subvariety of codimension at least $2$ in an algebraic manifold\footnote{An algebraic manifold is a smooth algebraic variety over $\C$, by definition quasi-compact in the Zariski topology.  We take a subvariety to be closed and not necessarily irreducible.} that is Zariski-locally isomorphic to $\C^n$.  (A similar class was introduced in \cite[\S 3.5.D]{Gromov1989}.)  The subclass $\A_0$ of algebraic manifolds Zariski-locally isomorphic to $\C^n$ contains, for example, $\C^n$ itself, complex projective spaces, all Grassmannians, all compact rational surfaces, all smooth complete toric varieties, and any vector bundle over a manifold in $\A_0$.  (Our definitions of $\A_0$ and $\A$ are more general than \cite[Definition 6.4.5]{Forstneric2011}; see Remark \ref{r:quasi-proj-not-needed}.)  For more examples of manifolds of class $\A_0$, see \cite[Section 4]{APS2014} (where the term \textit{A-covered} is used).

A challenging open question in basic Oka theory is whether the Oka property for, say, projective manifolds is a birational invariant.  In other words, how can you say what it means for a complex manifold to be bimeromorphically equivalent to an Oka manifold $Y$ without mentioning $Y$?  We do not know.  Our understanding of the interaction of the Oka property with the operation of blowing up a submanifold, even just a point, is still very limited.  The following result is due to Gromov (\cite[\S 3.5.D'']{Gromov1989}; see also \cite[Proposition 6.4.7]{Forstneric2011} and \cite[Section 4, Statement (9)]{APS2014}).

\begin{theorem*}[Gromov]
A manifold of Class $\A$ blown up at finitely many points is of Class $\A$ and hence Oka.
\end{theorem*}

Forstneri\v c proved that $\C^n$ blown up at each point of a tame discrete set is Oka \cite[Proposition 6.4.11]{Forstneric2011}.  It follows that a complex torus of dimension at least $2$, blown up at finitely many points, is Oka \cite[Corollary 6.4.12]{Forstneric2011}.  We are not aware of any other previous results about blow-ups of Oka manifolds being Oka.

Our main result is a strengthening of Gromov's theorem.

\begin{main-theorem}
A manifold of Class $\A$ blown up along any algebraic submanifold (not necessarily connected) is Oka.
\end{main-theorem}

We do not tackle the Oka property directly, but instead verify a geometric sufficient condition for it to hold, called algebraic subellipticity.  (This is how manifolds of Class $\A$ are shown to be Oka.)  An algebraic manifold is algebraically subelliptic if it has a finite dominating family of algebraic sprays \cite[Definition 5.5.11]{Forstneric2011}.  Algebraic subellipticity is a very interesting property for the following reasons.
\begin{itemize}
\item  It is (obviously) a purely algebraic property, but \ldots
\item  \ldots it has massive analytic consequences (namely the Oka property).
\item   It satisfies a localisation principle (due to Gromov \cite[\S 3.5.B]{Gromov1989}; see also \cite[Proposition 6.4.2]{Forstneric2011}), which sometimes offers the only way to the Oka property, for example here and in \cite[Proposition 4.10]{Hanysz2014}.  There is no known holomorphic analogue of this principle.
\item  It implies several algebraic Oka-type properties \cite[Sections 7.8 and 7.10]{Forstneric2011}.  For example, if $X$ is an affine algebraic variety and $Y$ is an algebraically subelliptic manifold, then a holomorphic map $X\to Y$ is approximable by regular maps, uniformly on compact subsets of $X$, if and only if it is homotopic to a regular map.
\end{itemize}

The bulk of this paper is devoted to the proof of the following result.

\begin{theorem}  \label{t:main}
Let $S$ be an algebraic subvariety of $\C^n$, $n\geq 2$, of codimension at least $2$.  The blow-up of $\C^n\setminus S$ along an algebraic submanifold is algebraically subelliptic.
\end{theorem}

By localisation of algebraic subellipticity, the following corollary is immediate, and implies our main theorem.

\begin{corollary}  \label{c:locally-affine-subelliptic}
The blow-up of a manifold of class $\mathscr A$ along an algebraic submanifold is algebraically subelliptic.
\end{corollary}

\begin{remark}  \label{r:quasi-proj-not-needed}
In Forstneri\v c's monograph, the localisation principle for algebraic subellipticity is proved under the assumption that the algebraic manifold $Y$ in question is quasi-projective \cite[Proposition 6.4.2]{Forstneric2011}.  This assumption is only used to ensure that for every point $y\in Y$ and every algebraic subvariety $Z$ of $Y$ with $y\notin Z$, there is an algebraic hypersurface $H$ in $Y$ with $Z\subset H$ but $y\notin H$.  By \cite[Theorem 4.1]{Borelli1963}, every algebraic manifold has this property, so the quasi-projectivity assumption is not needed.
\end{remark}

Next we present two corollaries of the fact that $\C^n$ blown up along an algebraic submanifold is Oka.

The first result confirms a prediction of the conjecture that every open Riemann surface may be properly holomorphically embedded into $\C^2$.  This is the remaining unresolved case of Forster's famous conjecture \cite[p.~183]{Forster1970}.  Let $A$ be an open Riemann surface embedded in $\C^n$ (such an embedding exists for every $n\geq 3$).  If there is an embedding $f:A\to\C^2$, then $f$ extends to a holomorphic map $F:\C^n\to\C^2$, and $F^{-1}(f(A))$ either is, or (if $F^{-1}(f(A))=\C^n$) contains, a hypersurface in $\C^n$ containing $A$ that retracts holomorphically onto $A$.  When $A$ is algebraic, Corollary \ref{c:retract} below confirms that $A$ is indeed a hypersurface retract.

By \cite[proof of Proposition 12 and Remark 13]{FL2014}, if $A$ is a connected analytic submanifold of $\C^n$, every holomorphic vector bundle over $A$ is holomorphically trivial, the blow-up $B$ of $\C^n$ along $A$ is Oka, and every continuous map $A\to B$ is null-homotopic, then $A$ is a holomorphic retract of a smooth analytic hypersurface in $\C^n$.  This result, Theorem \ref{t:main}, and the observation that $B$ is simply connected yield the following corollary.

\begin{corollary}  \label{c:retract}
Let $A$ be a connected algebraic submanifold of $\C^n$.  If $A$ is a curve or $A$ is contractible, then $A$ is a holomorphic retract of a smooth analytic hypersurface in $\C^n$.
\end{corollary}

As far as we know, there are contractible affine algebraic manifolds $A$ that are not known to be a hypersurface, for example Ramanujam's surface $R$ and products such as $R\times R$ and $R\times\C^k$.  For such $A$, the corollary is nontrivial.

One of the dozen or more nontrivially equivalent formulations of the Oka property says that a complex manifold $Y$ is Oka if for every Stein manifold $X$ with a subvariety $S$, a holomorphic map $S\to Y$ has a holomorphic extension $X\to Y$ if it has a continuous extension.  The second result follows from Theorem \ref{t:main} and the universal property of the blow-up; the details are given in Section \ref{sec:other-proofs}.

\begin{corollary}  \label{c:many-maps}
Let $A$ be an algebraic submanifold of $\C^n$, $n\geq 2$, $A\neq \C^n$, and let $T$ be a discrete subset of $\C^m$, $m\geq 1$, or a smooth analytic curve in $\C^m$, $m\geq 2$.  Let $f:T\to\C^n$ be holomorphic (an arbitrary map if $T$ is discrete).  Then $f$ extends to a holomorphic map $F:\C^m\to\C^n$ such that $F^{-1}(A)$ is a hypersurface.
\end{corollary}

We interpret the corollary to mean that there are \textit{many} holomorphic maps $\C^m\to\C^n$ that pull $A$ back to a hypersurface.

We now turn to a weaker, simpler property for which we can obtain stronger results.  An algebraic manifold $X$ is said to be algebraically dominable at a point $x$ in $X$ if there is a regular map $f:\C^n\to X$ such that $f(0)=x$ and $f$ is a local isomorphism at $0$.  We say that $X$ is algebraically dominable if it is algebraically dominable at some point, and that $X$ is strongly algebraically dominable if it is dominable at every point.

We use the technology of composed sprays and the Quillen-Suslin theorem to prove the following result.

\begin{proposition}  \label{p:quillen-suslin}
An algebraically subelliptic manifold is strongly algebraically dominable.
\end{proposition}

The next corollary is then immediate.

\begin{corollary}  \label{c:locally-affine-dom}
The blow-up of a manifold of class $\mathscr A$ along an algebraic submanifold is strongly algebraically dominable.
\end{corollary}

Note that if a projective manifold is algebraically dominable, then it is unirational and hence rationally connected.  We do not know any examples of algebraic manifolds that are dominable but not algebraically subelliptic, but it seems unlikely that the two properties are equivalent.  Strong dominability is not known to imply the Oka property.

Using Theorem \ref{t:main} and Proposition \ref{p:quillen-suslin}, we establish the following result.

\begin{proposition}  \label{p:stronger-dom}
The blow-up of $\C^n$, $n\geq 2$, along a closed subscheme $A$ is algebraically dominable at every point over the complement of the singular locus of $A$.
\end{proposition}

A closed subscheme of $\C^n$ is nothing but an ideal in the coordinate ring $\C[x_1,\ldots,x_n]$.

Finally, we are able to show that algebraic dominability is preserved by an arbitrary blow-up with a smooth centre.  The analogous result for algebraic subellipticity is beyond our reach for now.

\begin{theorem}  \label{t:general-strong-dom}
Let $B$ be the blow-up of an algebraic manifold $X$ along an algebraic submanifold.  If $X$ is algebraically dominable at a point $x$, then $B$ is algebraically dominable at every point over $x$.  Hence, if $X$ is algebraically dominable, so is $B$, and if $X$ is strongly algebraically dominable, so is $B$.
\end{theorem}

Let us mention the related result that if $X$ is uniformly rational (meaning that $X$ is covered by open sets isomorphic to open subsets of affine space), then so is $B$ (\cite[\S 3.5.E]{Gromov1989}, \cite[Proposition 2.6]{BB2014}).

In the next section we prove Theorem \ref{t:main}.  In the final section we prove Corollary \ref{c:many-maps}, Proposition \ref{p:quillen-suslin}, Proposition \ref{p:stronger-dom}, and Theorem \ref{t:general-strong-dom}.

\section{Proof of Theorem \ref{t:main}}
\label{sec:proof-of-main-theorem}

\subsection{}
This section is devoted to the proof of our main result, Theorem \ref{t:main}.  We start by proving the theorem in case $S=\varnothing$.  Let $B$ be the blow-up of $\C^n$, $n\geq 2$, along an algebraic submanifold $A$ of $\C^n$ (not necessarily connected) with exceptional divisor $E\subset B$.  Write $\pi$ for the projection $B\to\C^n$.  Without loss of generality we may assume that each component of $A$ has codimension at least $2$.  We will show that $B$ is algebraically subelliptic.  By Gromov's localisation principle, it suffices to show that $B$ can be covered by Zariski-open sets $U$ carrying regular sprays $\C^s\times U\to B$ that together dominate at each point $b$ of $B$.  Now $B\setminus E$ is isomorphic to $\C^n\setminus A$, which, as shown by Gromov (\cite[\S 0.5.B(iii)]{Gromov1989}, \cite[Proposition 5.5.14]{Forstneric2011}), is algebraically elliptic (with some high value of $s$).  Thus we take $b\in E$.  The sprays constructed below all have $s=1$.

Let $a=\pi(b)\in A$.  We may take $a$ to be the origin in $\C^n$.  Viewing $E$ as the projectivised normal bundle of $A$, we can represent $b$ by a vector $v\in T_a\C^n\setminus T_a A$.  The kernel of the tangent map $d_b\pi:T_b B\to T_a \C^n$ is the subspace $T_b\pi^{-1}(a)$ of dimension $\codim_a A-1$.  The image of $d_b\pi$ is $\C v\oplus T_a A$.  We first construct sprays that span the kernel.  Then we give a different construction of sprays that span some vector (that we have not tried to pin down) over a generic vector in the image.  This suffices to prove the theorem.

Let $r=\codim_a A\geq 2$.  After a linear change of coordinates, $T_a A\subset T_a\C^n\cong\C^n$ is given by the equations $x_1,\ldots,x_r=0$.  Then, in a Zariski neighbourhood $U$ of $a$ in $\C^n$, $A$ is the common zero locus of polynomials $u_1,\ldots,u_r$ with $u_j(x)=x_j+\textrm{higher order terms}$.  We can take $\C^n\setminus U$ to consist of the components of $A$ other than the component $A_0$ containing $a$ (call their union $A_1$) and of the common zeros of $u_1,\ldots,u_r$ other than $A_0$.  By removing from $U$ a subvariety of $A_0$ not containing $a$, we may assume that $d_x u_1,\ldots, d_x u_r$ are linearly independent for all $x\in A\cap U$.  We view $\pi^{-1}(U)\subset B$ as the closure in $U\times\P^{r-1}$ of the set 
\[ \{(x,\lambda)\in (U\setminus A)\times\P^{r-1}: \lambda=[u_1(x),\ldots,u_r(x)]\}. \]
In other words, $\pi^{-1}(U)$ is the graph of the rational map $[u_1,\ldots,u_r]:U\to\P^{r-1}$.  The map $\pi$ is the projection onto the first factor.  Note that $\pi^{-1}(U)$ is covered by $r$ affine Zariski-open sets of the same form, one of which is
\[ Y = \{ (x,\lambda)\in U\times\C^{r-1}:u_j(x)=\lambda_j u_r(x), j=1,\ldots,r-1\}. \]
Note also that $u_r\circ\pi$ is a defining function for $E\cap Y$ as a submanifold of $Y$.  We may assume that $b\in Y$.  Let $\tilde B$ be the graph of the rational map $[u_1,\ldots,u_r]:\C^n\to\P^{r-1}$ and $\tilde\pi:\tilde B\to\C^n$ be the projection.  The projection $\tilde\pi^{-1}(\C^n\setminus A_1)\to\pi^{-1}(\C^n\setminus A_1)$ is an isomorphism over $U$.

\subsection{}
To produce the first type of spray, we make use of the complete regular flows on $\C^n$ fixing $A$ pointwise, and therefore restricting to complete flows on $\C^n\setminus A$, that appear in Gromov's proof that $\C^n\setminus A$ is algebraically elliptic.  Define
\[ \phi:\C\times\C^n\to\C^n, \quad \phi(t,x)=x+th(\tau(x))\zeta, \]
where $\tau:\C^n\to\C^{n-1}$ is a surjective linear projection such that $\tau\vert A$ is proper, $\zeta\neq 0$ is in the kernel of $\tau$, and $h:\C^{n-1}\to\C$ is a polynomial which vanishes on the subvariety $\tau(A)$.  For a generic choice of $h$, $\tau$, $\zeta$, and $\xi\in T_b B$, we have:
\begin{itemize}
\item  $\eta=d_b\pi(\xi)\notin T_a A$.
\item  $\zeta\notin \C\eta+T_a A$.
\item  $d_b u_r(\eta)\neq 0$.
\item  $(d_{\tau(a)}h\circ d_a\tau)(\eta)\neq 0$.
\end{itemize}
Extend $\xi$ to a vector field (with the same name) on a small enough neighbourhood of $b$ in $E\cap Y$ that the above properties hold with $b$ replaced by a nearby $y\in E\cap Y$ and $a$ replaced by $\pi(y)$.

Define a regular map $f:\C\times Y\to\C^n\setminus A_1$ by the formula
\[ f(t,y)=\phi(t,\pi(y))=th(\tau(x))\zeta+x. \]
If $y=(x,\lambda)\in E\cap Y$, then $u_j(f(t,y))=u_j(x)=0$, so there are regular functions $\lambda_1,\ldots,\lambda_r$ on $\C\times Y$ such that $u_j(f(t,y))=u_r(x)\lambda_j(t,y)$ for $j=1,\ldots,r$ and $(t,y)\in\C\times Y$.  The map $f$ lifts to a rational map $F:\C\times Y\to\tilde\pi^{-1}(\C^n\setminus A_1)\subset\tilde B$ with
\[ F(t,y)= (f(t,y), [\lambda_1(t,y),\ldots,\lambda_r(t,y)]). \]
We claim that $F$ is regular on $\C\times V$ for some Zariski neighbourhood $V\subset Y$ of $b$.  

First, it is clear that $F$ is regular on $\C\times(Y\setminus E)$.  Next, for $F$ to be regular on $\C\times\{b\}$, we require $(\lambda_1(t,b),\ldots,\lambda_r(t,b))\neq (0,\ldots,0)$ for all $t\in\C$.  Differentiating the identity $u_j(f(t,y))=u_r(x)\lambda_j(t,y)$ with respect to $y$ at $(t,b)$ and evaluating the tangent maps at $\xi$ gives
\begin{equation}  \label{eq:derivatives}
d_a u_j\big(t(d_{\tau(a)}h\circ d_a\tau \circ d_b\pi)(\xi)\zeta+d_b\pi(\xi)\big) = \lambda_j(t,b) d_b(u_r\circ\pi)(\xi).
\end{equation}
The common kernel of $d_a u_1,\ldots,d_a u_r$ is $T_a A$, so our requirement is met if 
\[ t(d_{\tau(a)}h\circ d_a\tau \circ d_b\pi)(\xi)\zeta+\eta\notin T_a A \] 
for all $t\in\C$.  This holds since $\zeta\notin \C\eta+T_a A$ and $\eta\notin T_a A$.  Finally, we show that $F$ is regular on $\C\times\{y\}$ for $y\in E\cap Y$ sufficiently close to $b$.  Otherwise, there is a sequence $((t_\nu, y_\nu))$ with $y_\nu\in E\cap Y$, $y_\nu\to b$, and $\lambda_j(t_\nu, y_\nu)=0$ for $j=1,\ldots,r$.  We may assume that $t_\nu\to\infty$, for otherwise the inequality $(\lambda_1(t,b),\ldots,\lambda_r(t,b))\neq (0,\ldots,0)$ for all $t\in\C$ is contradicted.  Now (\ref{eq:derivatives}) holds with $b$ replaced by $y_\nu$ and $a$ by $\pi(y_\nu)\in A$, and $t=t_\nu$.  Letting $\nu\to\infty$, we conclude that $d_a u_j\big((d_{\tau(a)}h\circ d_a\tau)(\eta)\zeta\big)=0$ for $j=1,\ldots,r$, that is, $(d_{\tau(a)}h\circ d_a\tau)(\eta)\zeta\in T_a A$, which is ruled out by the generic choices made above.

Thus, postcomposing $F$ with the projection onto $\pi^{-1}(\C^n\setminus A_1)$, which is an isomorphism over $U$, yields a regular spray $G$ on $V\subset\pi^{-1}(U)$ with values in $\pi^{-1}(\C^n\setminus A_1)\subset B$.  Now $\dfrac{\partial f}{\partial t}(0,b)=0$, so $\dfrac{\partial G}{\partial t}(0,b)$ must lie in $\ker d_b\pi=T_b \pi^{-1}(a)$.  Differentiating (\ref{eq:derivatives}) with respect to $t$ at $(0,b)$ gives
\[ \dfrac{\partial \lambda_j}{\partial t}(0,b) d_a u_r(\eta) = (d_{\tau(a)}h\circ d_a\tau)(\eta) d_a u_j(\zeta). \]
By the choice of $u_1,\ldots,u_r$, $d_a u_j(\zeta)=\zeta_j$.  Hence the derivative at $0$ of the lifting $\C\to\C^r\setminus\{0\}$, $t\mapsto (\lambda_1(t,b),\ldots,\lambda_r(t,b))$, is
\[ \frac{(d_{\tau(a)}h\circ d_a\tau)(\eta)}{d_a u_r(\eta)} (\zeta_1,\ldots,\zeta_r). \]
This shows that we can produce $r-1$ sprays that span all of $T_b \pi^{-1}(a)$.

\subsection{}
We now turn to a different construction of sprays that span some vector over a generic vector in the image $\C v\oplus T_a A$ of $d_b\pi$.  

It is well known that every algebraic subvariety of $\C^n$ is a rational hypersurface retract.  Here, we restrict a linear projection $L:\C^n\to \C^{n-r+1}$ to $A_0$ and let $W=L^{-1}(L(A_0))$.  (Recall that $r=\codim_a A$.)  For generic $L$, the regular map $A_0\to L(A_0)$ is biregular at $a$, the hypersurface $W$ in $\C^n$ is smooth at $a$, and we have a rational retraction $W\to L(A_0)\to A_0$.  Thus, possibly after shrinking $U$, there is a hypersurface $W$ in $\C^n$ containing $A_0$ and smooth at $a$, with a regular retraction $\rho:W\cap U\to A\cap U$.  We may assume that any one of the polynomials $u_1,\ldots,u_r$, say $u_r$, is a defining function for $W$.  Let $V$ be the hypersurface $(W\cap U)\times\C^{r-1}$ in $U\times\C^{r-1}$.

Now $V$ is defined by the equation $u_r=0$ and $Y$ is defined by the equations $u_j=\lambda_j u_r$, $j=1,\ldots,r-1$.  Thus $V\cap Y=E\cap Y$.  Since $d_x u_1,\ldots, d_x u_r$ are linearly independent for all $x\in A\cap U$, we see that $V$ and $Y$ intersect transversely over $A\cap U$.

It is well known that the Zariski topology of a smooth algebraic variety has a basis consisting of open sets that are isomorphic to closed affine hypersurfaces (\cite[Theorem 5.7]{BMS1989}, \cite[Theorem 2.5]{Jelonek2000}).  We need a variant of this fact.

\noindent
\textbf{Claim.}  There is a Zariski neighbourhood $Z$ of $b$ in $U\times\C^{r-1}$ and a regular embedding $\gamma$ of $(V\cup Y)\cap Z$ as a closed subvariety of $\C^m$, $m=n+r-1$.

We take the claim for granted for now and prove it in the next subsection.  Write $V'=V\cap Z$ and $Y'=Y\cap Z$.  Because $\gamma(V')$ and $\gamma(Y')$ intersect transversely, the well-defined map $\gamma(V'\cup Y')\to\C^n$ defined on $\gamma(V')$ as $\rho\circ\pi\circ\gamma^{-1}$, and on $\gamma(Y')$ as $\pi\circ\gamma^{-1}$, is regular.  We extend this map to a regular map $\phi:\C^m\to\C^n$.  Then $\gamma(E\cap Y')\subset \gamma(V') \subset\phi^{-1}(A)$.

Let $I(A)$ be the defining ideal of $A$.  Next we show that $\phi^*I(A)$ is principal near $\gamma(b)$.  Let $p$ be a defining polynomial for $\gamma(V')$.  Then there are polynomials $q_1,\ldots,q_r$ such that
\[ u_j\circ\phi = p\, q_j, \qquad j=1,\ldots,r. \]
It suffices to show that $\gamma(E\cap Y')\cap\{q_1,\ldots,q_r=0\}$ is empty (so $\phi^{-1}(A)=\gamma(V')$ near $\gamma(E\cap Y')$).  For this, it is enough to find a tangent vector $w\in T_{\gamma(b)}\C^m$ such that 
\[ q_j(\gamma(b))d_{\gamma(b)} p(w)+p(\gamma(b)) d_{\gamma(b)} q_j(w) = d_{\gamma(b)}(u_j\circ\phi)(w)\neq 0 \]
for some $j\in\{1,\ldots,r\}$, since then $q_j(\gamma(b))\neq 0$.  Thus we need $d_{\gamma(b)}\phi(w)\notin T_a A$.  Now $d_b\pi(T_b Y)$ is larger than $T_a A$, so there is $w\in T_{\gamma(b)} \gamma(Y')$ with $d_{\gamma(b)}(\pi\circ\gamma^{-1})(w)\notin T_a A$.  Since $\phi=\pi\circ\gamma^{-1}$ on $\gamma(Y')$, we have $d_{\gamma(b)}\phi(w)=d_{\gamma(b)}(\pi\circ\gamma^{-1})(w)$.

Take $\zeta$ in $T_{\gamma(b)}\C^m$ (identified with $\C^m$ itself) and define a regular map
\[ f:\C\times Y'\to\C^n, \quad f(t,y)=\phi(\gamma(y)+t\zeta), \]
with $f(0,\cdot)=\pi$ on $Y'$.  Since $f^* I(A)$ is principal near $(0,b)$, the rational lifting $F:\C\times Y'\to B$ of $f$ is regular near $(0,b)$.  In fact, for generic $\zeta\in\C^m$, $F$ is regular on the product of $\C$ and some Zariski neighbourhood of $b$ in $Y'$.  Namely, let $Q$ be the subvariety of $\C^m$ where $\phi^*I(A)$ is not principal.  We need the line $\gamma(b)+\C\zeta$ to avoid $Q$, also at infinity in $\P^m$.  Since $\codim Q\geq 2$, this holds for generic $\zeta$.

Now $\dfrac{\partial f}{\partial t}(0,b)=d_{\gamma(b)}\phi(\zeta)$.  Since $d_{\gamma(b)}\phi(T_{\gamma(b)}\gamma(V'))=T_a A$, we have 
\[ d_{\gamma(b)}\phi(T_{\gamma(b)}\C^m)=d_b\pi(T_b Y). \]
Hence we obtain local sprays $F$ such that $\dfrac{\partial F}{\partial t}(0,b)$ lies over a generic vector in $d_b\pi(T_b Y)$, as desired.

\subsection{}
We conclude the proof of Theorem \ref{t:main} in case $S=\varnothing$ by proving the claim.  Our argument is based on Jelonek's proof of \cite[Theorem 2.5]{Jelonek2000}.

Let $\overline V=W\times\C^{r-1}$ and $\overline Y$ be the closure of $V$ and $Y$ in $\C^m$, respectively.  Then $R=(\overline V\cup\overline Y)\setminus U$ is a subvariety of codimension at least $2$ in $\C^m$.  Let $T$ be the union of $\overline V$ and a hypersurface containing $\overline Y$.  Then $T$ is a hypersurface in $\C^m$ with $b\in T$.  We will show that $b$ has a Zariski neighbourhood $Z$ in $\C^m$, disjoint from $R$, such that $T\cap Z$ embeds as a closed subvariety of $\C^m$.

After a generic change of coordinates of the form $x_j\mapsto x_j+a_j x_m$, $j=1,\ldots,m-1$, $x_m\mapsto x_m$, $T$ has a defining polynomial of the form
\[ x_m^k+\sum_{j=0}^{k-1}a_j(x_1,\ldots,x_{m-1})x_m^j=0. \]
Let $p:\C^m\to\C^{m-1}$ be the projection $(x_1,\ldots,x_m)\mapsto (x_1,\ldots,x_{m-1})$.  Then $p(R)$ is contained in a hypersurface in $\C^{m-1}$ defined by a polynomial $h$.  Let $H=\{x\in\C^m:x_m=0\}$ and $N=\{x\in\C^m:h(x_1,\ldots,x_{m-1})=0\}$.  We may assume that $0\notin T\cup N$ and $b\notin H\cup N$.  Let $R'=T\cap(H\cup N)$.  Then $R\subset R'$ and $Z=\C^m\setminus R'$ is a Zariski neighbourhood of $b$.  Define
\[ F:\C^m\to\C^m, \quad (x_1,\ldots,x_m)\mapsto (x_1,\ldots,x_{m-1},h(x_1,\ldots,x_{m-1})x_m). \]
Clearly, $F$ restricts to an automorphism of $\C^m\setminus N$.  Using the form of the defining polynomial of $T$, it is easy to show that
\[ \overline{F(T)} \cap N\subset H\cap N. \]
It follows that $\overline{F(T)}\setminus H = F(T)\setminus H$.  Since $F(N)\subset H$, we have 
\[ F(T)\setminus H=F(T\setminus N)\setminus H\subset \C^m\setminus N. \]
Hence $\overline{F(T)}\setminus H$ is isomorphic to 
\[ F^{-1}(F(T\setminus N)\setminus H)=T\setminus(H\cup N)=T\cap Z. \]
Now define
\[ \sigma:\C^m\to\C^m, \quad (x_1,\ldots,x_m)\mapsto (x_1 x_m,\ldots,x_{m-1}x_m,x_m). \]
Then $\sigma$ is an automorphism of $\C^m\setminus H$ and $\sigma^{-1}(H)=\sigma^{-1}(0)=H$.  Since $0\notin T\cup N$, we have $0\notin \overline{F(T)}$, so
\[ \sigma^{-1}(\overline{F(T)}) = \sigma^{-1}(\overline{F(T)}\setminus\{0\})=\sigma^{-1}(\overline{F(T)})\setminus H = \sigma^{-1}(\overline{F(T)}\setminus H).\]
We conclude that $T\cap Z$ is isomorphic to the closed subvariety $\sigma^{-1}(\overline{F(T)})$ of $\C^m$.

\subsection{}
Now let $S$ be an algebraic subvariety of $\C^n$, $n \geq 2$, of codimension at least $2$, and $A$ be an algebraic submanifold of $\C^n\setminus S$.  Let $B$ be the blow-up of $\C^n\setminus S$ along $A$.  We indicate how the proof above can be modified so as to show that $B$ is algebraically subelliptic.

We include $S$ in $\C^n\setminus U$.  In the definition of the map $\phi$ in the construction of the first type of spray, we replace $A$ by the union of $S$ and the closure of $A$ in $\C^n$.  The map $f$ then takes values in $\C^n\setminus (A_1\cup S)$ and the construction goes through.

In the definition of the map $f$ in the construction of the second type of spray, we replace $\gamma(y)+t\zeta$ by a flow that avoids $\phi^{-1}(S)$.  To obtain such a flow we need $\codim \phi^{-1}(S)\geq 2$, which must be built into the construction of $\phi$ as an extension.  To this end we use the following corollary of a theorem of Jelonek.

\begin{proposition}  \label{p:Jelonek-extension}
Let $m\geq n$, $X$ be an algebraic subvariety of $\C^m$, and $f:X\to\C^n$ be a polynomial map.  Then there is a polynomial map $F:\C^m\to\C^n$ extending $f$ such that $\dim F^{-1}(z)\setminus X\leq m-n$ for all $z\in\C^n$.
\end{proposition}

\begin{proof}
Embed $\C^n$ as $\C^n\times\{0\}$ in $\C^m$.  Then $f$ induces a map $\tilde f:X\to\C^m$, which extends to a polynomial map $\tilde F:\C^m\to\C^m$ such that $\tilde F\vert\C^m\setminus X$ has finite fibres \cite[Theorem 3.9]{Jelonek1997}.  Let $\pi:\C^m\to\C^n$, $(z_1,\ldots,z_m)\mapsto (z_1,\ldots,z_n)$.  Then $F=\pi\circ\tilde F$ is the desired map.
\end{proof}

\section{Other Proofs}
\label{sec:other-proofs}

\begin{proof}[Proof of Corollary \ref{c:many-maps}]
Let $\pi:B\to\C^n$ be the blow-up along $A$ and let $f:T\to\C^n$ be holomorphic.  First note that $f$ factors through $\pi$ by a holomorphic map $g:T\to B$.  This is clear if $T$ is discrete, so suppose that $T$ is a smooth analytic curve.  If $f(T)\not\subset A$, then the preimage of $A$ by $f$, as a complex subspace of $T$, is locally principal since $\dim T=1$, so by the universal property of the blow-up, $f$ factors through $\pi$.  If $f(T)\subset A$, we use the geometric construction of the blow-up.  The pullback by $f$ of the normal bundle of $A$ in $\C^n$ is holomorphically trivial, again since $\dim T=1$, and a nowhere-vanishing section of the pullback bundle over $T$ defines $g$.

Next we need an extension of $g:T\to B$ to a continuous map $\C^m\to B$.  If $T$ is discrete, this is elementary.  For example, take an injection $g_1:T\to\mathbb R$ and a continuous map $g_2:\mathbb R\to B$ such that $g=g_2\circ g_1$, and extend $g_1$ to a continuous map $\C^m\to\mathbb R$.  If $T$ is a smooth analytic curve, since $B$ is simply connected and $T$ is homotopy equivalent to a disjoint union of bouquets of circles, $g$ is homotopic to a continuous map $\tilde g:T\to B$ with a countable image.  It is easy to see that $\tilde g$ extends continuously to $\C^m$ (for example by factoring $\tilde g$ through $\mathbb R$ as above), so $g$ does as well.

Since $B$ is Oka, $g$ has a holomorphic extension $h:\C^m\to B$.  Let $F=\pi\circ h:\C^m\to \C^n$.  Then $F$ is a holomorphic extension of $f$ and $F^{-1}(A)=h^{-1}(\pi^{-1}(A))$ is a hypersurface -- except that $F^{-1}(A)$ might be empty or all of $\C^m$.  To avert the former, add an extra point or component to $T$ and let $f$ map it into $A$.  To avert the latter, add an extra point or component to $T$ and let $f$ map it outside of $A$.
\end{proof}

\begin{proof}[Proof of Proposition \ref{p:quillen-suslin}]
We refer to \cite[Section 6.3]{Forstneric2011} for Gromov's theory of composed sprays.  Let $X$ be an algebraic manifold with a dominating family of algebraic sprays $(E_j, \pi_j, s_j)$, $j=1,\ldots,m\geq 2$ (if $m=1$, there is nothing to prove).  The composed spray $(E_1\ast E_2, \pi_1\ast\pi_2, s_1\ast s_2)$ is defined as the pullback
\[ E_1\ast E_2 = \{ (e_1,e_2)\in E_1\times E_2:s_1(e_1)=\pi_2(e_2)\} \]
with
\[ \pi_1\ast\pi_2(e_1,e_2)=\pi_1(e_1), \quad s_1\ast s_2(e_1,e_2)=s_2(e_2). \]
Then $E_1\ast E_2$ is a vector bundle over $E_1$, and it has a natural zero-section over $X$, but we do not know whether it is a vector bundle, even holomorphically, over $X$.  Otherwise it is a spray over $X$ in the usual sense.  With that same proviso, we have a composed spray bundle $E=(\cdots(E_1\ast E_2)\ast\cdots)\ast E_m$, which is dominating over $X$.  Now $E$ is a vector bundle over a vector bundle over \ldots a vector bundle over $X$, so each fibre of $E$ is a vector bundle over a vector bundle over \ldots an affine space.  (Up to this point, the theory of composed sprays is the same in the algebraic category and the holomorphic category.)  We now invoke the Quillen-Suslin theorem, which states that every algebraic vector bundle over an affine space is algebraically trivial, and conclude that each fibre of $E$ is isomorphic to an affine space, which implies that $X$ is strongly algebraically dominable.
\end{proof}

\begin{proof}[Proof of Proposition \ref{p:stronger-dom}]
Let $A$ be a closed subscheme of $\C^n$, $n\geq 2$.  The defining ideal of $A$ is generated by polynomials $h_1,\ldots,h_m$ with greatest common divisor $h$.  The blow-up of $\C^n$ along $A$ is the same as the blow-up of $\C^n$ along the subscheme defined by the ideal generated by $h_1/h,\ldots,h_m/h$.  Thus we may assume that $A$ has codimension at least $2$.  In particular, the singular locus $Z$ of $A$ has codimension at least $2$.  By Theorem \ref{t:main}, the blow-up of $\C^n\setminus Z$ along $A\setminus Z$ is algebraically subelliptic and hence strongly algebraically dominable by Proposition \ref{p:quillen-suslin}.
\end{proof}

\begin{proof}[Proof of Theorem \ref{t:general-strong-dom}]
Let $B$ be the blow-up of an algebraic manifold $X$ along an algebraic submanifold $A$.  Suppose that $X$ is algebraically dominable at a point $x$ and let $y\in B$ lie over $x$.  Let $f:\C^n\to X$ be a regular map that takes $0$ to $x$ and is a local isomorphism at $0$.  Let $\widehat\C^n$ be the blow-up of $\C^n$ along the subscheme $f^*A$.  Then $0$ is not a singular point of $f^*A$.  Denote the blow-up projections by $\pi:B\to X$ and $p:\widehat\C^n\to\C^n$.  Let $F:\widehat\C^n\to B$ be the regular lifting of $f\circ p$ by $\pi$, taking a point $z$ over $0$ to $y$.  Then $F$ is a local isomorphism at $z$, so it suffices to show that $\widehat\C^n$ is dominable at $z$, but this follows from Proposition~\ref{p:stronger-dom}.
\end{proof}


\begin{thebibliography}{88}

\bibitem{APS2014}
I. Arzhantsev, A. Perepechko, H. S\"u\ss.  \textit{Infinite transitivity on universal torsors.}  J. London Math. Soc. (2) \textbf{89} (2014) 762--778.

\bibitem{BMS1989}
S. Bloch, M. P. Murthy, L. Szpiro.  \textit{Zero cycles and the number of generators of an ideal.}  Colloque en l'honneur de Pierre Samuel (Orsay, 1987).  M\'em. Soc. Math. France (N.S.) No. 38 (1989) 51--74. 

\bibitem{BB2014}
F. Bogomolov, C. B\"ohning.  \textit{On uniformly rational varieties.}  Topology, geometry, integrable systems, and mathematical physics, 33--48, 
Amer. Math. Soc. Transl. Ser. 2, 234.  Amer. Math. Soc., 2014. 

\bibitem{Borelli1963}
M. Borelli.  \textit{Divisorial varieties.}  Pacific J. Math. \textbf{13} (1963) 375--388.

\bibitem{Forster1970}
O. Forster.  \textit{Plongements des vari\'et\'es de Stein.} Comment. Math. Helv. \textbf{45} (1970) 170--184.

\bibitem{Forstneric2011}
F. Forstneri\v c.  \textit{Stein manifolds and holomorphic mappings.  The homotopy principle in complex analysis.}  Ergebnisse der Mathematik und ihrer Grenzgebiete, 3.\ Folge, 56.  Springer-Verlag, 2011.

\bibitem{Forstneric2013}
F. Forstneri\v c.  \textit{Oka manifolds: from Oka to Stein and back.}  With an appendix by F.\ L\'arusson.  Ann. Fac. Sci. Toulouse Math. (6) \textbf{22} (2013) 747--809.

\bibitem{FL2011} 
F. Forstneri\v c, F. L\'arusson.  \textit{Survey of Oka theory.}  New York J. Math. \textbf{17a} (2011) 1--28.

\bibitem{FL2014} 
F. Forstneri\v c, F. L\'arusson.  \textit{Holomorphic flexibility properties of compact complex surfaces.}  Int. Math. Res. Not. 2014, no. 13, 3714--3734. 

\bibitem{Gromov1989}
M. Gromov.  \textit{Oka's principle for holomorphic sections of elliptic bundles.}  J. Amer. Math. Soc. \textbf{2} (1989) 851--897.

\bibitem{Hanysz2014}
A. Hanysz.  \textit{Oka properties of some hypersurface complements.}
Proc. Amer. Math. Soc. \textbf{142} (2014) 483--496.

\bibitem{Jelonek1997}
Z. Jelonek.  \textit{A hypersurface which has the Abhyankar-Moh property.}  Math. Ann. \textbf{308} (1997) 73--84.

\bibitem{Jelonek2000}
Z. Jelonek.  \textit{Local characterization of algebraic manifolds and characterization of components of the set $S_f$.}  Ann. Polon. Math. \textbf{75} (2000) 7--13.

\end{thebibliography}
\end{document}